\documentclass[reqno,A4paper]{amsart}

\usepackage{amsmath}
\usepackage{amssymb}
\usepackage{amsthm}
\usepackage{enumerate}
\usepackage{mathrsfs}
\usepackage[colorinlistoftodos]{todonotes}
\usepackage[colorlinks=true, allcolors=blue]{hyperref}
\usepackage{xy}
\usepackage{cancel}
\usepackage{epstopdf}
\usepackage{verbatim}
\usepackage{amscd}
\usepackage{graphicx}
\usepackage{tikz}
\usetikzlibrary{cd}
\usepackage{subcaption}
\usepackage{amsthm}
\usepackage{pgf}
\usepackage{mathrsfs}
\usepackage[toc,page]{appendix}
\usepackage{caption}
\usepackage[colorinlistoftodos]{todonotes}

\makeatletter
\newtheorem*{rep@theorem}{\rep@title}
\newcommand{\newreptheorem}[2]{%
\newenvironment{rep#1}[1]{%
 \def\rep@title{#2 \ref{##1}}%
 \begin{rep@theorem}}%
 {\end{rep@theorem}}}
\makeatother

\makeatletter
\newtheorem*{rep@corollary}{\rep@title}
\newcommand{\newrepcorollary}[2]{%
\newenvironment{rep#1}[1]{%
 \def\rep@title{#2 \ref{##1}}%
 \begin{rep@corollary}}%
 {\end{rep@corollary}}}
\makeatother

\newtheorem{theorem}{Theorem}[section]
\newtheorem{lemma}[theorem]{Lemma}
\newtheorem{corollary}[theorem]{Corollary}
\newtheorem{proposition}[theorem]{Proposition}
\newreptheorem{theorem}{Theorem}
\newrepcorollary{corollary}{Corollary}

\newtheoremstyle{customrem}
{3pt}
{3pt}
{}
{}
{\bfseries}
{.}
{.5em}
{}
\theoremstyle{customrem}
\newtheorem{remark}[theorem]{Remark}

\usepackage{eqlist}
\usepackage{array}

\numberwithin{equation}{section}

\begin{document}

\title[Tribonacci numbers and primes of the form $p=x^2+11y^2$.]%
{Tribonacci numbers and primes of the form $p=x^2+11y^2$.}
\author[Tim Evink \and Paul Alexander Helminck]%
{Tim Evink* \and Paul Alexander Helminck**}

\newcommand{\acr}{\newline\indent}

\address{\llap{*\,}Johann Bernoulli Instituut voor Wiskunde en Informatica \acr
                   Faculteit Wiskunde en Natuurwetenschappen (FWN)\acr
                   Nijenborgh 9\acr
                   9747 AG Groningen\acr
                   NETHERLANDS}
\email{t.evink.1@student.rug.nl}

\address{\llap{**\,}ALTA, Institute for Algebra, Geometry, Topology and their Applications\acr
                    University of Bremen,\acr
                    Postfach 330440 \acr
                    D-28334, Bremen\acr
                    GERMANY}
\email{helminck@uni-bremen.de}
\begin{abstract}
In this paper we show that for any prime number $p$ not equal to $11$ or $19$, the Tribonacci number $T_{p-1}$ is divisible by $p$ if and only if $p$ is of the form $x^2+11y^2$. We first use class field theory on the Galois closure of the number field corresponding to the  polynomial $x^3-x^2-x-1$ to give the splitting behavior of primes in this number field. After that, we apply these results to the explicit exponential formula for $T_{p-1}$. We also give a connection between the Tribonacci numbers and the Fourier coefficients of the unique newform of weight $2$ and level $11$. 
\end{abstract}
\maketitle
\section{Introduction}
Let $(T_{n})$ be the Tribonacci sequence, defined by the recurrence relation 
\begin{equation}
T_{n+3}=T_{n+2}+T_{n+1}+T_{n}
\end{equation}
and the initial values $T_{0}=0$, $T_{1}=1$ and $T_{2}=1$. 
In this paper we will prove the following theorem:
\begin{theorem}\label{MainTheorem}
Let $p$ be a prime number not equal to $11$ or $19$. Then $T_{p-1}$ is divisible by $p$ if and only if $p=x^2+11y^2$ for $x,y\in\mathbb{Z}$.
\end{theorem}
The extra conditions given by $p\neq 11,19$ are indeed necessary, since $T_{10}=149$ is not divisible by $11$ and $11$ is of the form $x^2+11y^2$, and $T_{18}=19513=19\cdot{1027}$ is divisible by $19$, even though $19$ is not of the form $x^2+11y^2$. The authors are under the impression that this second exception has no deeper number-theoretic relevance and deem it to be an accident of sorts. Indeed, its sudden appearance in the proof of Theorem \ref{MainTheorem} will probably make this point of view clear. The appearance of $11$ in this Theorem however \emph{is} natural, as it is one of the ramifying primes in the extension $\mathbb{Q}\subset{\mathbb{Q}(\alpha)}$, where $\alpha$ is a root of $x^3-x^2-x-1$. 

There is an analogous statement for Fibonacci sequences, whose proof can be found in \cite[Lehrsatz, Page 76]{crelle1846journal} or \cite[Volume 1, Chapter XVII]{Dickson1}. For a prime number $p$ not equal to $2$ or $5$, it states that $F_{p-1}$ is divisible by $p$ if and only if $5$ is a square in $\mathbb{F}_{p}$, which happens exactly when $p\equiv{1,4}\bmod{5}$. Note that $5$ 
is the only ramifying prime in the extension $\mathbb{Q}\subset{\mathbb{Q}(\sqrt{5})}$. 

The proof of Theorem \ref{MainTheorem} uses class field theory to give the splitting behavior of primes in the number field $\mathbb{Q}(\alpha)$, where $\alpha$ is a zero of $f=x^3-x^2-x-1$. Here we find that $p$ splits completely in $\mathbb{Q}(\alpha)$ if and only if $p=x^2+11y^2$. In other words, $f\text{ mod }{p}$ has three distinct zeros in $\mathbb{F}_{p}$ if and only if $p=x^2+11y^2$. We then use this criterion on the known explicit formula for the Tribonacci sequence in terms of the zeros of $f=x^3-x^2-x-1$. These formulas are exponential of nature, allowing us to reduce them using Frobenius. 

For the \emph{Tetranacci numbers}, defined by $T_{n+4}=T_{n+3}+T_{n+2}+T_{n+1}+T_{n}$ and $T_{0}=0$, $T_{1}=1$, $T_{2}=1$ and $T_{3}=2$, the resulting polynomial $f=x^{4}-x^{3}-x^{2}-x-1$ has Galois group $S_{4}$. Since $S_{4}$ is solvable, this means that one can again use class field theory on the intermediate fields. For the \emph{Pentanacci} numbers (which are defined in a similar way), the resulting polynomial has Galois group $S_{5}$, which is not solvable, meaning that one cannot use class field theory here.

We have also included a connection between Theorem \ref{MainTheorem} and the Fourier coefficients of the unique newform of weight $2$ and level $11$ in Section \ref{Modularforms}. This newform is given explicitly by
\begin{equation}
f=q{\prod_{n=1}^{\infty}(1-q^n)^2(1-q^{11n})^2},
\end{equation}
where $q=e^{2\pi{i}\tau}$. The idea here is that the fixed field of the kernel of the representation of the absolute Galois group on the $2$-torsion of the modular curve $X_{0}(11)$ is the same as the number field $\mathbb{Q}(\alpha)$. Using some results on elliptic curves, this then gives a $\bmod\,{2}$-criterion on the Fourier coefficients of $f$ for a prime to split in the number field $\mathbb{Q}(\alpha)$.
The final result is as follows.
\begin{corollary}
\label{CorollaryMain}
Let $p$ be a prime not equal to $2,11$ or $19$, let $T_{n}$ be the $n$-th Tribonacci number and let $a_{n}(f)$ be the $n$-th Fourier coefficient of the unique newform of weight $2$ and level $11$. Then the following are equivalent:
\begin{enumerate}
\item $T_{p-1}\equiv{0}\bmod{p}$.
\item There exist $x$ and $y$ in $\mathbb{Z}$ such that $p=x^2+11y^2$.
\item $a_{p}(f)\equiv{0}\bmod{2}$ and $p\equiv{1,3,4,5,9}\bmod{11}$.
\end{enumerate} 
\end{corollary}

We conclude the paper by interpreting Corollary \ref{CorollaryMain} in terms of {\it{Serre's modularity conjecture}} and representation theory, see Remarks \ref{DifferenceCurves} and \ref{remprep} respectively.  


%

\section{Splitting behavior}\label{SectionSplittingBehavior}

Let $\alpha$ be a root of $f=x^3-x^2-x-1$ and consider the cubic number field $E=\mathbb{Q}(\alpha)$. The polynomial $f$ has discriminant $-44=-2^2\cdot 11$, and since $f\equiv (x-1)^3\bmod 2$ and $f(1)=-2\notin 2^2\mathbb{Z}$, we see that $\mathbb{Z}[\alpha]$ is regular and totally ramified over $2$ by 
\cite[Theorem 3.1]{Ste2}. In particular we obtain $\mathbb{Z}[\alpha]=\mathcal{O}_E$ and $\Delta_E=\Delta(\mathcal{O}_E)=-44$.

The splitting field of $f$ over $\mathbb{Q}$ is $L=E(\sqrt{-11})$. With $K=\mathbb{Q}(\sqrt{-11})$, we have that $L/K$ is a cyclic extension of degree $3$ and we will apply class field theory to this extension.
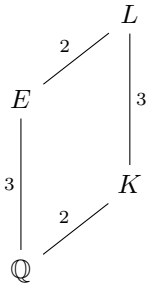
\begin{center}
\begin{tikzcd}[every arrow/.append style={dash}]
& L \\
E \arrow{ur}{2} &  \\
& K \arrow[swap]{uu}{3} \\
\mathbb{Q} \arrow{uu}{3} \arrow{ur}{2}& 
\end{tikzcd}
\end{center}
We wish to compute the conductor $\mathfrak{f}=\mathfrak{f}_{L/K}$. Any prime $p\neq 2,11$ is unramified in $E$ and hence also in its normal closure $L$, so $\mathrm{ord}_{\mathfrak{p}}(\mathfrak{f})=0$ for any prime $\mathfrak{p}$ of $K$ not lying over $2$ or $11$.\\

Now $2$ is inert in $K$ as $-11\equiv 5\bmod 8$ and it is totally ramified in $E$. As the ramification index is multiplicative in towers of field extensions, we conclude that the prime $(2)$ of $K$ ramifies in $L$. We then see that $(2)$ is tamely ramified in $L/K$ and thus
$\mathrm{ord}_{(2)}\mathfrak{f}=1$. Let $\mathfrak{p}_{11}$ denote the unique prime over $11$ in $K$. 
If $\mathfrak{p}_{11}$ were to ramify in $L$, then $11$ would be totally ramified in $E$. This would imply that the different $\mathfrak{D}_{E}$ is divisible by the square of the prime over $11$ in $E$ and this would then yield $\mathrm{ord}_{11}(\Delta_E)\geq 2$, which is not the case. One could also just factorize $f$ modulo $11$ to see that $11$ is not totally ramified. We conclude that $\mathfrak{p}_{11}$ is unramified in $L$. As $K$ is complex there are no infinite primes dividing $\mathfrak{f}$, so we conclude that $\mathfrak{f}=(2)$.


This discussion implies that $L$ is contained in the ray class field $H_{\mathfrak{f}}$. The degree $[H_{\mathfrak{f}}:K]$ is the order of the ray class group $Cl_{\mathfrak{f}}$, which fits inside the exact sequence
\[
0\xrightarrow{}(\mathcal{O}/2)^{*}/\mathrm{im}[\mathcal{O}^{*}]\xrightarrow{}Cl_{\mathfrak{f}}\xrightarrow{} Cl_{K}\xrightarrow{} 0,
\]
where $\mathcal{O}=\mathcal{O}_K$. As $Cl_K=1$, $(\mathcal{O}/2)^{*}$ has order $3$ and $|Cl_{\mathfrak{f}}|$ divides $3$, we conclude that $Cl_{\mathfrak{f}}$ has order $3$, hence $L=H_{\mathfrak{f}}=H_2$. We now prove
\begin{theorem}\label{class}
With the notation as above, a prime $p\neq 11$ splits completely in $L$ if and only if $p=x^2+11y^2$ for some $x,y\in\mathbb{Z}$.
\end{theorem}
\begin{proof}
Since $L=H_{\mathfrak{f}}$ we have the Artin reciprocity law 
\[
Cl_{\mathfrak{f}}\xrightarrow{\sim}\mathrm{Gal}(H_{\mathfrak{f}}/K),
\]
which tells us that a prime $\mathfrak{p}\neq(2)$ of $K$ splits completely in $L$ if and only if $\mathfrak{p}=(x)$ for some $x\equiv 1\bmod 2\mathcal{O}_K$, see for example \cite[Theorem 1.12]{Ste1}. This will be the main ingredient of the proof.\\

$(\Rightarrow)$ Suppose that $p$ splits completely in $L$. Then $p$ splits in $K$ so there is a prime $\mathfrak{p}$ in $K$ of norm $p$ that splits in $L$. Because of the Artin reciprocity law, this implies that $\mathfrak{p}=(1+2a)$ for some $a\in\mathcal{O}_K$. But $2a\in\mathbb{Z}[\sqrt{-11}]$, so we can write $1+2a=x+y\sqrt{-11}$ for $x,y\in\mathbb{Z}$.  We then  conclude
\[
p=N_{K/\mathbb{Q}}(\mathfrak{p})=|N_{K/\mathbb{Q}}(1+2a)|=x^2+11y^2.
\]
$(\Leftarrow)$ Suppose that $p=x^2+11y^2$. We then have $p\mathcal{O}_K=\mathfrak{p}\mathfrak{q}$ for $\mathfrak{p}=(x+y(2a-1))$ and $\mathfrak{q}=(x-y(2a-1))$, where $a=\tfrac{1+\sqrt{-11}}{2}$. Clearly $\mathfrak{p}$ and $\mathfrak{q}$ are not $1$ so they both have norm $p$. In particular they are prime ideals. They are distinct because $p\neq 11$ so $p$ is split in $K$.\\
Reducing $p=x^2+11y^2$ modulo $2$ we see that $x+y\equiv 1\bmod 2$. Reducing the generators of $\mathfrak{p}$ and $\mathfrak{q}$ modulo $2\mathcal{O}_K$ then yields
\[
x\pm y(2a-1)\equiv x+y+2ay\equiv 1\bmod 2\mathcal{O}_K,
\]
so we conclude from the Artin reciprocity law that $\mathfrak{p}$ and $\mathfrak{q}$ split in $L$, hence $p$ splits completely in $L$.
\end{proof}
Note that a prime $p\neq 11$ splits completely in $E$ if and only if splits completely in the normal closure $L$. As $\mathcal{O}_E=\mathbb{Z}[\alpha]$, we see that the Kummer-Dedekind theorem yields that $x^3-x^2-x-1$ has three distinct roots in $\mathbb{F}_p$ if and only if $p$ splits completely in $E$. Thus the content of Theorem \ref{class} can be rephrased by saying that for a prime $p\neq 11$ we have
\[
x^3-x^2-x-1\text{ has three distinct roots modulo }p\quad\Leftrightarrow\quad p=x^2+11y^2\text{ for some }x,y\in\mathbb{Z}.
\]

\section{Tribonacci numbers}

The Tribonacci sequence $(T_{n})$
is defined by the recurrence relation
\begin{equation}\label{RecRel}
T_{n+3}=T_{n+2}+T_{n+1}+T_{n},
\end{equation}
with initial values $T_{0}=0$, $T_{1}=1$ and $T_{2}=1$. We label the three roots of $x^3-x^2-x-1$ by $\{\alpha,\beta,\gamma\}$. We then easily see that the space $V$ of all $L$-valued sequences $(x_{n})_{n\geq{0}}$ such that $x_{n+3}=x_{n+2}+x_{n+1}+x_{n}$ is a vector space over $L$. Moreover, we have the following 
\begin{lemma}\label{VectorSpace}
$V$ is a three-dimensional vector space over $L$ with basis $\{(\alpha^{n}),(\beta^{n}),(\gamma^{n})\}$.
\end{lemma}
\begin{proof}
Every sequence $(x_{n})$ satisfying \ref{RecRel} is uniquely determined by its initial values $x_{0},x_{1}$ and $x_{2}$ in $L$. Since we are free to choose these, we see that $V$ is three-dimensional. To see that the proposed sequences form a basis, let $e_{1},e_{2},e_{3}$ be the unique sequences in $V$ determined respectively by the initial values  
$(1,0,0)$, $(0,1,0)$ and $(0,0,1)$. These three sequences form a basis of $V$ and in terms of this basis 
our sequences are given by the Vandermonde matrix
\begin{equation}\label{Vandermonde}A=
 \begin{pmatrix}
1 & 1 & 1 \\
\alpha &\beta & \gamma\\
\alpha^2 & \beta^2 & \gamma^2 
\end{pmatrix},
\end{equation}
which has determinant $(\alpha-\beta)(\alpha-\gamma)(\gamma-\beta)\neq{0}$. In other words, the sequences $(\alpha^{n}),(\beta^{n}),(\gamma^{n})$ form a basis for $V$, as desired. 
\end{proof}
It now follows that we can write $T_{n}=c_{1}\alpha^{n}+c_{2}\beta^{n}+c_{3}\gamma^{n}$ for $c_{i}\in{L}$. The following Lemma explicitly determines these $c_{i}$. 
\begin{lemma}\label{Constants}
We have
\begin{equation}
c_{1}=\dfrac{\alpha}{(\alpha-\beta)(\alpha-\gamma)},\,\,
c_{2}=\dfrac{\beta}{(\beta-\alpha)(\beta-\gamma)},\,\,
c_{3}=\dfrac{\gamma}{(\gamma-\alpha)(\gamma-\beta)}.
\end{equation}
\end{lemma}
\begin{proof}
Plugging in $T_{0}=0$, $T_{1}=1$ and $T_{2}=1$ in the equation $T_{n}=c_{1}\alpha^{n}+c_{2}\beta^{n}+c_{3}\gamma^{n}$, we obtain the linear system
\begin{align*}
&c_{1}+c_{2}+c_{3}=0\\
&c_{1}\alpha+c_{2}\beta+c_{3}\gamma=1,\\
&c_{1}\alpha^2+c_{2}\beta^{2}+c_{3}\gamma^{2}=1.
\end{align*} 
Using Cramer's rule, we then easily find the proposed solutions, as desired. 
\end{proof}
Let us write $$\delta=(\alpha-\beta)(\alpha-\gamma)(\beta-\gamma).$$ Note that $\delta^{2}=\Delta_{E}$ and $\delta=-\mathrm{det}(A)$, where $A$ is the Vandermonde matrix from Equation \ref{Vandermonde}. Then \begin{equation}\label{Formula12}
\delta\cdot{}T_{n}=\alpha^{n+1}(\beta-\gamma)-\beta^{n+1}(\alpha-\gamma)+\gamma^{n+1}(\alpha-\beta).
\end{equation}
Note that both sides of Equation \ref{Formula12} contain elements of $\mathcal{O}_{L}$, the ring of integers in $L$. This allows us to perform modular arithmetic on $T_{n}$ inside $\mathcal{O}_{L}$.

\begin{reptheorem}{MainTheorem}
Let $p$ be a prime number not equal to $11$ or $19$. Then $T_{p-1}$ is divisible by $p$ if and only if $p=x^2+11y^2$ for some $x,y\in\mathbb{Z}$.
\end{reptheorem}
\begin{proof}
Let $\mathfrak{p}$ be a prime of $L$ lying over $p$. For $p=2$, we clearly have that the theorem holds. We now suppose that $p\neq{2}$. Note that the elements occurring in Equation \ref{Formula12} are elements in $\mathcal{O}_{L}$, which we can reduce modulo $\mathfrak{p}$.
Doing this for $n=p-1$ yields 
\begin{equation}\label{LOL}
\delta\cdot{}T_{p-1}\equiv{\alpha^{p}(\beta-\gamma)-\beta^{p}(\alpha-\gamma)+\gamma^{p}(\alpha-\beta)}\bmod{\mathfrak{p}}.
\end{equation}

$(\Rightarrow)$ Suppose that $p=x^2+11y^2$. Theorem \ref{class} implies that $p$ splits completely in $L$, i.e. the decomposition group $D_{\mathfrak{p}}$ is trivial. The lift of the Frobenius automorphism of $\mathcal{O}_L/\mathfrak{p}=\mathbb{F}_p$ is then the identity so that
\begin{align*}
\delta\cdot{}T_{p-1}&\equiv{\alpha^{p}(\beta-\gamma)-\beta^{p}(\alpha-\gamma)+\gamma^{p}(\alpha-\beta)}\bmod{\mathfrak{p}}\\
&={\alpha(\beta-\gamma)-\beta(\alpha-\gamma)+\gamma(\alpha-\beta)}\bmod{\mathfrak{p}}\\
&=0\bmod{\mathfrak{p}}.
\end{align*}
Since $p$ is unramified in $E$ and thus in $L$, we have $\delta\not\equiv 0\bmod{\mathfrak{p}}$. We conclude that $T_{p-1}\equiv{0}\bmod{p}$.

$(\Leftarrow)$ Conversely, suppose that $T_{p-1}\equiv 0\bmod p$. We will show that $D_{\mathfrak{p}}=(1)$.  
Since $p$ is unramified in $L$, we have that the decomposition group $D_{\mathfrak{p}}$ is cyclic. This implies that $D_{\mathfrak{p}}\neq{S_{3}}$. We will now rule out $|D_{\mathfrak{p}}|=2$ and $|D_{\mathfrak{p}}|=3$.

Suppose that $|D_{\mathfrak{p}}|=2$ and let $\sigma$ be the automorphism given by $\sigma(\alpha)=\beta$ and $\sigma(\gamma)=\gamma$. As Frobenius elements are conjugate, we may replace $\mathfrak{p}$ with $\tau\mathfrak{p}$ for some automorphism $\tau$ if necessary so that $\sigma=\mathrm{Frob}_{\mathfrak{p}}$. This means that $\alpha^{p}\equiv{\beta}\bmod{\mathfrak{p}}$, $\beta^{p}\equiv{\alpha}\bmod{\mathfrak{p}}$ and $\gamma^{p}\equiv{\gamma}\bmod{\mathfrak{p}}$. Since we also have $T_{p-1}\equiv 0\bmod\mathfrak{p}$ by assumption, Equation \ref{LOL} yields
\begin{equation}
\beta^{2}-\alpha^{2}+2\gamma(\alpha-\beta)\equiv{0}\bmod{\mathfrak{p}}.
\end{equation}
Since $\beta\not\equiv{\alpha}\bmod{\mathfrak{p}}$, we find that $\beta+\alpha\equiv{}2\gamma\bmod{\mathfrak{p}}$. With $\alpha+\beta+\gamma=1$, this results in $\gamma\equiv{}1/3\bmod{\mathfrak{p}}$. We also obtain $\alpha\beta\equiv{}3\bmod{\mathfrak{p}}$ and $\alpha+\beta\equiv{}2/3\bmod{\mathfrak{p}}$. We then have
\begin{equation}
-1=\alpha\beta+\alpha\gamma+\beta\gamma=\alpha\beta+(\alpha+\beta)\gamma\equiv 3+1/3\cdot{2/3}\bmod{\mathfrak{p}}=29/9\bmod{\mathfrak{p}}.
\end{equation}
We thus obtain $38\equiv{}0\bmod{\mathfrak{p}}$, which contradicts our assumption $p\neq{2,19}$. 

Now assume that $|D_{\mathfrak{p}}|=3$ and let $\sigma$ be the automorphism determined by
$\sigma(\alpha)=\beta$, $\sigma(\beta)=\gamma$ and $\sigma(\gamma)=\alpha$. Just as in the previous case we may assume without loss of generality that $\sigma=\mathrm{Frob}_{\mathfrak{p}}$. This gives $\alpha^p\equiv\beta\bmod\mathfrak{p}$, $\beta^{p}\equiv\gamma\bmod\mathfrak{p}$ and $\gamma^p\equiv\alpha\bmod\mathfrak{p}$. Using this on Equation \ref{LOL} results in 
\begin{equation}
\alpha^{2}+\beta^{2}+\gamma^{2}\equiv\alpha\beta+\alpha\gamma+\beta\gamma\bmod{\mathfrak{p}}=-1\bmod{\mathfrak{p}}.
\end{equation}
We use this to obtain 
\begin{equation}
1^2=(\alpha+\beta+\gamma)^2=\alpha^2+\beta^2+\gamma^2+2\alpha\beta+2\alpha\gamma+2\beta\gamma\equiv-3\bmod{\mathfrak{p}}.
\end{equation}
In other words, $p=2$, a contradiction. We conclude that $D_{\mathfrak{p}}$ is indeed trivial.\qedhere
\end{proof}

\section{A connection with modular forms}\label{Modularforms}

In this section, we give the splitting behavior of primes in the number field $\mathbb{Q}(\alpha)$ with $\alpha^3-\alpha^2-\alpha-1=0$ using the unique newform of weight $2$ and level $11$. In the process, we will demonstrate a particular case of Serre's conjecture on odd two-dimensional Galois representations, see Remark \ref{DifferenceCurves}. This approach was suggested to us by Prof. Jaap Top. For more background material, we refer the reader to \cite{Silv1} for an introduction to elliptic curves and \cite{ModularForms2016} for an introduction to modular forms. We would also like to refer the reader to \cite{Shimura1966}, where reciprocity laws similar to the one in Proposition \ref{ModularProp} were obtained using the unique newform of weight $2$ and level $11$. 	

Throughout this section, $G_{\mathbb{Q}}=\mathrm{Gal}(\overline{\mathbb{Q}}/\mathbb{Q})$ will denote the absolute Galois group of $\mathbb{Q}$.
Let $A$ be the elliptic curve defined by the Weierstrass equation
\begin{equation}\label{EquationWeierstrass}
A:y^2+y=x^3-x^2-10x-20.
\end{equation}
This elliptic curve has conductor $11$ (it has semistable reduction at $11$ and good reduction at the other primes) and is in fact isomorphic to the modular curve $X_{0}(11)$, as shown in \cite{TWe1}. We will write $X_{0}(11)_{\mathbb{C}}=X_{0}(11)\times_{\mathrm{Spec}(\mathbb{Q})}{\mathrm{Spec}(\mathbb{C})}$ for the base change of $X_{0}(11)$ to $\mathrm{Spec}(\mathbb{C})$ and $X_{0}(11)^{\mathrm{an}}$ for the complex analytification of $X_{0}(11)_{\mathbb{C}}$ as in \cite{Serre1956}. Simply put, the equations that define $X_{0}(11)_{\mathbb{C}}$ also canonically define a compact Riemann surface and we denote that Riemann surface by $X_{0}(11)^{\mathrm{an}}$. We will write $\Omega(X_{0}(11)^{\mathrm{an}})$ for the sheaf of $1$-forms on the Riemann surface $X_{0}(11)^{\mathrm{an}}$ and $H^{0}(\Omega(X_{0}(11)^{\mathrm{an}}))$ for its global sections. These global sections are also known as the {\it{holomorphic differentials}} on $X_{0}(11)^{\mathrm{an}}$. 
By \cite[Corollary 2.17]{shimura1971}, 
there is an isomorphism 
\begin{equation}
\mathcal{S}_{2}(\Gamma_{0}(11))\rightarrow{H^{0}(\Omega(X_{0}(11)^{\mathrm{an}}))}.
\end{equation}
Since $X_{0}(11)^{\mathrm{an}}$ is an elliptic curve, this last space is one dimensional. To be explicit, it is generated as a $\mathbb{C}$-vector space by the holomorphic 1-form
\begin{equation}
\omega:=\dfrac{dx}{2y+1},
\end{equation}
where $x$ and $y$ are as in Equation \ref{EquationWeierstrass}. For a proof that $\omega$ is holomorphic, see \cite[Chapter 3, Proposition 1.5]{Silv1}. We find that there is exactly one {\it{normalized}} cusp form of weight $2$ and level $11$. Explicitly, it is given by
\begin{equation}\label{EquationModular}
f:=q{\prod_{n=1}^{\infty}(1-q^n)^2(1-q^{11n})^2}=\eta(\tau)^2\cdot\eta(11\tau)^2,
\end{equation}
where $q=e^{2\pi{i}\tau}$ and $\eta(\tau)$ is the Dedekind eta function. See \cite[Proposition 3.2.2]{ModularForms2016} for a more general result on these twisted products of eta functions.
Viewing $f$ as a power series in $q$, we write $f=\sum_{n=1}^{\infty}a_{n}q^{n}$. Note that we can recover the $a_{n}$ by considering finitely many terms in Equation \ref{EquationModular} (although this is a highly inefficient way of calculating them). Throughout this section, we will also use the notation $a_{n}(g)$ for the $n$-th coefficient of any newform $g$. \\

Let $a_{p}(A)=p+1-\#A(\mathbb{F}_{p})$, where $A$ is as in Equation \ref{EquationWeierstrass}. We have the following theorem, whose proof is mostly based on the famous {\it{Eichler\textendash{}Shimura relation}}, see \cite[Theorem 8.7.2]{ModularForms2016}.
\begin{theorem}
Let $X_{0}(N)\rightarrow{E}$ be a non\textendash{}constant morphism over $\mathbb{Q}$ from the modular curve $X_{0}(N)$ to an elliptic curve $E$, where $E$ has conductor $N_{E}$ (see \cite[Section 8.3]{ModularForms2016}). Then for some newform $g\in{\mathcal{S}_{2}(\Gamma_{0}(M_{g}))}$ where $M_{g}|N$, we have
\begin{equation}
a_{p}(g)=a_{p}(E)
\end{equation}
for all $p\nmid{N\cdot{}N_{E}}$.
\end{theorem}
\begin{proof}
See \cite[Theorem 8.8.2]{ModularForms2016}. We would also like to direct the reader to \cite[Section 6]{Shimura1966}, where this result was obtained for the unique newform $f$ of weight $2$ and level $11$ and the elliptic curve $A$ considered in Equation \ref{EquationWeierstrass}.
\end{proof}
\begin{corollary}\label{ModularityCorollary}
Let $f$ be as in Equation \ref{EquationModular} and let $A$ be the elliptic curve from Equation \ref{EquationWeierstrass}. Then for any prime number $p\neq{11}$, we have that
\begin{equation}
a_{p}(f)=a_{p}(A).
\end{equation}
\end{corollary}

Let us now consider the elliptic curve $B$ defined by the Weierstrass equation
\begin{equation}\label{EllipticCurve2}
B:y^2=x^3-x^2-x-1.
\end{equation}
Consider the representation
\begin{equation}
\overline{\rho}_{B,2}:G_{\mathbb{Q}}\rightarrow{\mathrm{Aut}(B[2])}
\end{equation} of $G_{\mathbb{Q}}$ on the $2$-torsion of $B$. This representation is intimately connected to our number field $\mathbb{Q}(\alpha)$ and the corresponding Galois closure $L$. 
Let us explain this in more detail. The kernel of $\overline{\rho}_{B,2}$ corresponds to a finite field extension of $\mathbb{Q}$, which is obtained by adding the $x$ and $y$ coordinates of the non-infinite $2$-torsion points to $\mathbb{Q}$, see \cite[Section 6.3]{Silverman2015}. For any elliptic curve defined by a Weierstrass equation of the form $y^2=h(x)$, the $2$-torsion points are defined by the points with $y$ coordinate equal to $0$ by \cite[Theorem 2.1]{Silverman2015}, so we see that there is a natural bijection between the roots of $x^3-x^2-x-1$ and the set of nontrivial $2$-torsion points on $B$. Furthermore, the group $G_{\mathbb{Q}}$ acts on these $2$-torsion points $(\alpha,0)$ by
\begin{equation}
\sigma:(\alpha,0)\mapsto{(\sigma(\alpha),0)},
\end{equation}
so we obtain an isomorphism
\begin{equation}
G_{\mathbb{Q}}/\mathrm{ker}(\overline{\rho}_{B,2})\simeq{\mathrm{Gal}(L/\mathbb{Q})}\simeq{S_{3}}.
\end{equation}
We have that $G_{\mathbb{Q}}/\mathrm{ker}(\overline{\rho}_{B,2})$ injects into $\mathrm{Aut}(B[2])$, so by comparing orders we obtain $G_{\mathbb{Q}}/\mathrm{ker}(\overline{\rho}_{B,2})\simeq{\mathrm{Aut}(B[2])}$. We then obtain an identification of $\mathrm{Gal}(L/\mathbb{Q})$ with $\mathrm{GL}_{2}(\mathbb{F}_{2})$ by choosing a basis of $B[2]$. In other words, we can identify elements of the Galois group of $L/\mathbb{Q}$ with matrices in $\mathrm{GL}_{2}(\mathbb{F}_{2})$ after choosing a basis.

We now have the following Lemma:
\begin{lemma}\label{Isomorphism}
There exists an isomorphism of $G_{\mathbb{Q}}$-modules $A[2]\simeq{B[2]}$.
\end{lemma}
\begin{proof}
We will follow \cite[Section 3.1]{Silv1} to transform the long Weierstrass equation in Equation \ref{EquationWeierstrass} into a short Weierstrass equation.
We complete the square in Equation \ref{EquationWeierstrass} with $y'=y+1/2$ and obtain
\begin{equation}
y'^2=x^3-x^2-10x-20+1/4.
\end{equation}
Taking $x'=x-1/3$, we then obtain
\begin{equation}
y'^2=(x')^3 - 31/3\cdot{}x' - 2501/108.
\end{equation}
Using one final set of transformations given by $y''=2^3\cdot{3^3}\cdot{}y'$ and $x''=2^2\cdot{3^2}\cdot{x'}$ leads to the short Weierstrass form
\begin{equation}
y^2=x^3 - 13392{}x -1080432,
\end{equation}
where $y=y''$ and $x=x''$. As we saw before the Lemma, the $2$-torsion subgroup $X[2]$ of any elliptic curve $X$ in Weierstrass form $y^2=h(x)$ consists of the three points with $y$-coordinate equal to $0$ together with the point at infinity. We will use this to construct an isomorphism of $G_{\mathbb{Q}}$-modules $B[2]\rightarrow{A[2]}$. 

Consider the splitting fields $L_{h_{1}}$ and $L_{h_{2}}$ of the polynomials $h_{1}(x):=x^3-x^2-x-1$ and $h_{2}(x):=x^3 - 13392{}x -1080432$. By definition, we have $L_{h_{1}}=L$, where $L$ is as defined in Section \ref{SectionSplittingBehavior}. Writing $\alpha$ for a root of $x^3-x^2-x-1$, we easily find that $\alpha':=72\alpha^2-18\alpha-66$ satisfies $h_{2}(\alpha')=0$. This implies that $L_{h_{2}}=L$.  
We now explicitly give the desired isomorphism of ${G}_{\mathbb{Q}}$-modules $\phi:B[2]\rightarrow{A[2]}$. 
It is given by 
\begin{equation}
\phi:(\sigma(\alpha),0)\mapsto{(\sigma(\alpha'),0)}
\end{equation}
for any $\sigma\in{\mathrm{Gal}(L/\mathbb{Q})}$. Note that these are well-defined by the fact that the splitting fields of $h_{1}$ and $h_{2}$ are the same. An easy check then shows that $\phi$ is a bijective homomorphism that commutes with the action of $G_{\mathbb{Q}}$. We conclude that $\phi$ is an isomorphism of $G_{\mathbb{Q}}$-modules, as desired. 
\end{proof}

Now let $p$ be any prime not equal to $2,11$. Consider the $2$-adic Tate module $T_{2}(A)$ as defined in \cite[Section 3.7]{Silv1}. Choosing a $\mathbb{Z}_{2}$-basis $\{P_{1},P_{2}\}$ of $T_{2}(A)$, we obtain a representation 
\begin{equation}
\rho_{A,2}:G_{\mathbb{Q}}\rightarrow{\mathrm{GL}_{2}(\mathbb{Z}_{2})}.
\end{equation}
The images of $P_{1}$ and $P_{2}$ under the map $T_{2}(A)\rightarrow{A[2]}$ then give rise to a representation
\begin{equation}
\overline{\rho}_{A,2}:G_{\mathbb{Q}}\rightarrow{\mathrm{GL}_{2}(\mathbb{F}_{2})},
\end{equation}
which fits in a commutative diagram with the reduction map $r:\mathrm{GL}_{2}(\mathbb{Z}_{2})\rightarrow{\mathrm{GL}_{2}(\mathbb{F}_{2})}$ as follows:
\begin{equation}
\label{DiagramCommutative}
\begin{tikzcd}
G_{\mathbb{Q}} \arrow{r}{\rho_{A,2}} \arrow{dr}{\overline{\rho}_{A,2}} & \mathrm{GL}_{2}(\mathbb{Z}_{2}) \arrow{d}{r} \\
{} & \mathrm{GL}_{2}(\mathbb{F}_{2}).
\end{tikzcd}
\end{equation}


Let $D_{p}$ be a decomposition group in $G_{\mathbb{Q}}$ and let $\mathrm{Frob}_{p}$ be a Frobenius element in $D_{p}$. That is, $\mathrm{Frob}_{p}$ is an element of ${G}_{\mathbb{Q}}$ such that its image in the quotient
\begin{equation}
D_{p}/I_{p}\simeq{\mathrm{Gal}(\overline{\mathbb{F}}_{p}/\mathbb{F}_{p})}
\end{equation}
is the Frobenius automorphism $x\mapsto{x^{p}}$. We note that the restriction of $\mathrm{Frob}_{p}$ to $L$ (which is the Galois closure of $\mathbb{Q}(\alpha)/\mathbb{Q}$) is a Frobenius element above $p$ of the Galois group of $L$ over $\mathbb{Q}$. 
We now have the following two standard facts on elliptic curves and number fields:
\begin{lemma}\label{EllipticCurveLemma}
Let $p$ be a prime number not equal to $2$ or $11$. Then 
$$\mathrm{tr}(\overline{\rho}_{A,2}(\mathrm{Frob}_{p}))\equiv{}a_{p}(A)\bmod{2}.$$ 
\end{lemma}	
\begin{proof}
By \cite[Theorem 9.4.1]{ModularForms2016}, we have that
\begin{equation}
\mathrm{tr}(\rho_{A,2}(\mathrm{Frob}_{p}))=a_{p}(A).
\end{equation}
Using the commutative diagram in Equation \ref{DiagramCommutative}, we then obtain
\begin{equation}
\mathrm{tr}(\overline{\rho}_{A,2}(\mathrm{Frob}_{p}))\equiv{}a_{p}(A)\bmod{2}.
\end{equation}
This finishes the proof.

\end{proof}
\begin{lemma}\label{EasyLemma}
$\overline{\rho}_{A,2}(\mathrm{Frob}_{p})=\mathrm{Id}$ if and only if $p$ splits completely in $L$.
\end{lemma}
\begin{proof}
This follows from the earlier mentioned fact that the restriction of a Frobenius element is again a Frobenius element. 
\end{proof}
Let $\chi:(\mathbb{Z}/11\mathbb{Z})^{*}\rightarrow{\{\pm{1}\}}$ be the unique nontrivial quadratic character. It is given by 
\begin{align*}
\chi(\overline{n})&=1 \text{ if }n\equiv{1,3,4,5,9}\bmod{11},\\
\chi(\overline{n})&=-1 \text{ if }n\equiv{2,6,7,8,10}\bmod{11}.
\end{align*}
Since $\mathbb{Q}(\sqrt{-11})\subseteq{\mathbb{Q}({\zeta_{11}})}$, we obtain that $p$ splits completely in $\mathbb{Q}(\sqrt{-11})$ if and only if $\chi(\overline{p})=1$. 
We then obtain the following criterion for $p$ to split completely in $L$:
\begin{proposition}\label{ModularProp}
A prime $p$ not equal to $2$ or $11$ splits completely in $L$ if and only if $a_{p}(f)=a_{p}(A)\equiv{0}\bmod{2}$ and $p\equiv{1,3,4,5,9}\bmod{11}$. 
\end{proposition}
\begin{proof}
Suppose that $a_{p}(f)\equiv{0}\bmod{2}$ and $p\equiv{1,3,4,5,9}\bmod{11}$. Using Lemma \ref{EllipticCurveLemma} and Corollary \ref{ModularityCorollary}, we see that $\mathrm{tr}(\overline{\rho}_{A,2}(\mathrm{Frob}_{p}))\equiv{}0\bmod{2}$.  The only invertible $2\times{2}$-matrices over $\mathbb{F}_{2}$ having trace zero are given by
\begin{equation}\label{MatricesTrace}
\begin{pmatrix}
1 & 0 \\
0 & 1 
\end{pmatrix}, \begin{pmatrix}
0 & 1 \\
1 & 0 
\end{pmatrix},
\begin{pmatrix}
1 & 0 \\
1 & 1 
\end{pmatrix},
\begin{pmatrix}
1 & 1 \\
0 & 1 
\end{pmatrix}.
\end{equation}
The non-identity matrices all correspond to elements of the Galois group that restrict to nontrivial elements of $\mathrm{Gal}(\mathbb{Q}(\sqrt{-11})/\mathbb{Q})$. Using $\chi(\overline{p})=1$, we see that $\overline{\rho}_{A,2}(\mathrm{Frob}_{p})$ must be the identity matrix.  
By Lemma \ref{EasyLemma}, we conclude that $p$ splits completely in $L$, as desired.

Conversely, suppose that $p$ splits completely in $L$. Then the Frobenius element above $p$ is the identity. In other words, $\overline{\rho}_{A,2}(\mathrm{Frob}_{p})=\mathrm{Id}$, which yields the equation $\mathrm{tr}(\overline{\rho}_{A,2}(\mathrm{Frob}_{p}))\equiv{}0\bmod{2}$. Since $p$ splits completely in $L$, it also splits completely in $K$, meaning that $p\equiv{1,3,4,5,9}\bmod{11}$ by the quadratic reciprocity law, as desired. 
\end{proof}
\begin{repcorollary}{CorollaryMain}
Let $p$ be a prime not equal to $2,11$ or $19$, let $T_{n}$ be the $n$-th Tribonacci number and let $a_{n}(f)$ be the $n$-th Fourier coefficient of the unique newform of weight $2$ and level $11$. Then the following are equivalent:
\begin{enumerate}
\item $T_{p-1}\equiv{0}\bmod{p}$.
\item There exist $x$ and $y$ in $\mathbb{Z}$ such that $p=x^2+11y^2$.
\item $a_{p}(f)\equiv{0}\bmod{2}$ and $p\equiv{1,3,4,5,9}\bmod{11}$.
\end{enumerate} 
\end{repcorollary}

\begin{remark}\label{DifferenceCurves}{\it{(Serre's conjecture)}}
We now explain our choice for $A$ and $B$. We chose $B$ because the $x$ coordinates of the $2$-torsion of $B$ define the Galois closure $L$ of the number field $K$ over $\mathbb{Q}$. Our choice for $A$ has to do with Serre's conjecture on odd irreducible two-dimensional continuous representations $\rho:G_{\mathbb{Q}}\rightarrow{\mathrm{GL}_{2}(\overline{\mathbb{F}}_{p})}$, which we now recall. For more background information, we refer the reader to \cite[Chapter 1]{SteinRibet1}, \cite[Section 9.6]{ModularForms2016} and \cite[Chapter 2]{CompGalois}. 

Let $f=\sum_{n\geq{1}}a_{n}q^{n}$ be a normalized eigenform in the space $\mathcal{S}_{k}(\Gamma_{1}(N))$ of complex weight-$k$ cusp forms on the subgroup $\Gamma_{1}(N)$ of $\mathrm{SL}_{2}(\mathbb{Z})$, where $k\geq{2}$ and $N\geq{1}$. We thus have that $a_{1}=1$ and $f|T_{n}=a_{n}\cdot{f}$ for all $n\geq{1}$, where $T_{n}$ is the $n$-th Hecke operator. There furthermore is a character $\chi:(\mathbb{Z}/N\mathbb{Z})^{*}\rightarrow{\mathbb{C}^{*}}$ that describes the action of the diamond operator $\langle{d}\rangle$ on $f$ by $ f|\langle{d}\rangle=\chi(d)\cdot{f}$. We will write $S_{k}(\Gamma_{1}(N),\chi)$ for the cusp forms that have character $\chi$. Let $\mathbf{K}_{f}$ be the field generated by the Fourier coefficients $a_{n}$ of $f$ over $\mathbb{Q}$. This is a finite extension of $\mathbb{Q}$, see \cite[Page 234]{ModularForms2016}. We will write $\mathcal{O}_{\mathbf{K}_{f}}$ for the ring of integers in $\mathbf{K}_{f}$. For any maximal ideal $\lambda$ in $\mathcal{O}_{\mathbf{K}_{f}}$, we write $\mathbf{K}_{f,\lambda}$ for the completion of $\mathbf{K}_{f}$ with respect to the absolute value induced by $\lambda$ and $\mathcal{O}_{\mathbf{K}_{f,\lambda}}$ for the corresponding ring of integers. We now have the following theorem.
\begin{theorem}\label{RepresentationEigenform}
Let $f\in{S_{k}(\Gamma_{1}(N),\chi)}$ be a normalized eigenform with number field $\mathbf{K}_{f}$. Let $\ell$ be a prime. For each maximal ideal $\lambda$ of $\mathcal{O}_{\mathbf{K}_{f}}$ lying over $\ell$, there is an irreducible continuous two-dimensional Galois representation
\begin{equation}
\rho_{f,\lambda}:G_{\mathbb{Q}}\rightarrow{\mathrm{GL}_{2}(\mathbf{K}_{f,\lambda})}.
\end{equation}
This representation is unramified at primes $p\nmid{\ell{N}}$. For any such $p$, let  $\mathrm{Frob}_{p}\in{G_{\mathbb{Q}}}$ be a Frobenius automorphism. Then the characteristic polynomial of $\rho_{f,\lambda}(\mathrm{Frob}_{p})$ is given by
\begin{equation}\label{Characteristic}
x^2-a_{p}(f)x+\chi(p)p^{k-1}.
\end{equation}
\end{theorem}
\begin{proof}
See \cite[Theorem 9.6.5]{ModularForms2016} and \cite[Theorem 2.4.1]{CompGalois}. 
\end{proof}
Note that Equation \ref{Characteristic} implies that 
$\mathrm{tr}(\rho_{f,\lambda}(\mathrm{Frob}_{p}))=a_{p}(f)$ and $\mathrm{det}(\rho_{f,\lambda}(\mathrm{Frob}_{p}))=\chi(p)p^{k-1}$.
By a simple compactness argument as given in \cite[Proposition 9.3.5]{ModularForms2016}, we can find a basis of the $\mathbf{K}_{f,\lambda}$-vector space $V:=(\mathbf{K}_{f,\lambda})^{2}$ such that $\rho_{f,\lambda}$ is similar to a representation $\rho'_{f,\lambda}:G_{\mathbb{Q}}\rightarrow{\mathrm{GL}_{2}(\mathcal{O}_{\mathbf{K}_{f,\lambda}})}$. We again write $\rho_{f,\lambda}$ for this representation. We then automatically obtain a reduced representation
\begin{equation}
\overline{\rho}_{f,\lambda}:G_{\mathbb{Q}}\rightarrow{\mathrm{GL}_{2}(\mathcal{O}_{\mathbf{K}_{f,\lambda}}/\lambda\mathcal{O}_{\mathbf{K}_{f,\lambda}})}.
\end{equation}
Viewing the residue field $\mathcal{O}_{\mathbf{K}_{f,\lambda}}/\lambda\mathcal{O}_{\mathbf{K}_{f,\lambda}}$ as a subfield of $\overline{\mathbb{F}}_{\ell}$, we thus obtain a representation $\overline{\rho}'_{f,\lambda}:G_{\mathbb{Q}}\rightarrow{\mathrm{GL}_{2}(\overline{\mathbb{F}}_{\ell})}$. We will also denote this representation by $\overline{\rho}_{f,\lambda}$. 

We now consider the "inverse problem" of the construction above. Let $\rho:G_{\mathbb{Q}}\rightarrow{\mathrm{GL}_{2}(\overline{\mathbb{F}}_{\ell})}$ be any odd, continuous, irreducible representation. By odd, we mean that the the image of a complex conjugation automorphism $c\in{G_{\mathbb{Q}}}$ under $\mathrm{det}\circ{\rho}$ is equal to $-1$. We can attach a level $N(\rho)$ to $\rho$ as follows (see \cite[Section 1.3]{SteinRibet1} and \cite[Section 3.1]{Cais2007}). It is a product
\begin{equation}
N(\rho)=\prod_{p\neq{\ell}}p^{n_{p}},
\end{equation}
where the $n_{p}$ are defined as follows. Let $I_{p}$ be an inertia group above $p$ and let $V=\overline{\mathbb{F}}_{\ell}^{2}$ be the vector space on which $G_{\mathbb{Q}}$ acts. Let
\begin{equation}
V^{I_{{p}}}=\{v\in{V}:\rho(\sigma)(v)=v\text{ for all }\sigma\in{I_{p}}\}.
\end{equation}
Then 
\begin{equation}
n_{p}=\mathrm{dim}(V/V^{I_{p}})+\mathrm{Swan}(V),
\end{equation}
where $\mathrm{Swan}(V)$ is the wild part of $\rho$, which is given by
\begin{equation}
\mathrm{Swan}(V)=\sum_{i=0}^{\infty}\dfrac{1}{[G_{0}:G_{i}]}\cdot{\mathrm{dim}(V/V^{G_{i}})}.
\end{equation}
Here $G_{0}=I_{p}$ and the $G_{i}$ are the higher ramification groups. The weight $k(\rho)$ corresponding to $\rho$ is somewhat harder to define, so we refer the reader to \cite[Chapter 2]{SteinRibet1}, \cite[Section 5.4]{Cais2007} and \cite[\S{}2]{Ser1987}. 

We can now state Serre's conjecture on odd irreducible two-dimensional representations of $G_{\mathbb{Q}}$:
\begin{theorem}\label{Serreconjecture}{\it{(Serre's conjecture)}}
Let $\rho:G_{\mathbb{Q}}\rightarrow{\mathrm{GL}_{2}(\overline{\mathbb{F}}_{\ell})}$ be an odd, continuous, irreducible representation of $G_{\mathbb{Q}}$. Then $\rho$ is isomorphic to $\overline{\rho}_{f,\lambda}$ for some $f$ and $\lambda$ as in Theorem \ref{RepresentationEigenform}. Furthermore, the weight and level of $f$ are equal to $k(\rho)$ and $N(\rho)$ respectively.
\end{theorem} 
\begin{proof}
This was proved by Khare and Wintenberger in \cite{Khare2009} and \cite{Khare20092}.
\end{proof}

Let us now state the relevance of Theorem \ref{Serreconjecture} in finding the normalized eigenform $f$ as in Equation \ref{EquationModular} and the elliptic curve $A$. Let $B$ be the elliptic curve as defined in Equation \ref{EllipticCurve2}. We will write $\rho$ for the representation 
$\overline{\rho}_{B,2}:G_{\mathbb{Q}}\rightarrow{\mathrm{GL}_{2}(\mathbb{F}_{2})}$ induced by the $2$-torsion points of $B$. The fixed field of the kernel of $\rho$ is then equal to $L$. 
By Theorem \ref{Serreconjecture}, there exists some normalized eigenform $f$ with a maximal ideal $\lambda$ of $\mathcal{O}_{\mathbf{K}_{f}}$ such that $\rho_{f,\lambda}$ is isomorphic to $\rho$. Let us calculate the level $N(\rho)$. Note that for any prime not equal to $2$ or $11$, the representation is unramified and thus $n(p)=0$. By unramified, we mean that the image of any inertia group $I_{p}$ is trivial under $\rho_{B,2}$.  By definition, we have that $N(\rho)$ contains no powers of $2$, so $N(\rho)$ is a power of $11$. More explicitly, let $I_{11}$ be an inertia group above $11$. Then its image in $\mathrm{GL}_{2}(\mathbb{F}_{2})$ has order $2$, since $11$ ramifies in the quadratic extension $\mathbb{Q}\subset{K}$ and not in the cubic extension $K\subset{L}$. Furthermore, there is exactly one $2$-torsion point $P$ that it fixes (one can invoke the orbit-stabilizer theorem from group theory here), so we find that $\mathrm{dim}(V^{I_{11}})=1$ and consequently $n(11)=1$. Indeed, the higher ramification groups are all trivial, so the Swan conductor is trivial.

The weight $k(\rho)$ can be calculated as follows. We refer the reader to \cite[Section 5.3]{Cais2007} for the definition and the relevant concepts.  The image $\rho(I_{2})$ of an inertia group $I_{2}$ above $2$ has order three, see our calculations in Section \ref{SectionSplittingBehavior}. In $\mathrm{GL}_{2}(\mathbb{F}_{2^{2}})$, we can then write the restriction as
\begin{equation}
\rho|_{I_{2}}=
\begin{pmatrix}
\chi_{3} & 0 \\
0 & \chi_{3}^{2} 
\end{pmatrix},
\end{equation}
where $\chi_{3}$ is a cyclotomic character of $\mathbb{F}_{2^{2}}^{*}$ of order three. That is to say, write $M=\rho(\sigma)$ for a $\sigma\in{I_{2}}$ mapping to an element of order three. Then $M$ satisfies $M^2+M+I=0$ by the Cayley\textendash{}Hamilton theorem from linear algebra. Let $\lambda\in{\mathbb{F}_{2^{2}}}$ satisfy $\lambda^2+\lambda+1=0$ and consider an element $v$ such that $Mv=\lambda{v}$. With respect to the basis $\{{v},\bar{v}\}$, $M$ is then given by
\begin{equation}
M=\begin{pmatrix}
\lambda & 0 \\
0 & \lambda^2 
\end{pmatrix}.
\end{equation}
Here $\bar{v}$ is the conjugate of $v$ under the action of $\mathrm{Gal}(\mathbb{F}_{2^{2}}/\mathbb{F}_{2})$. We now see that the restriction of the representation $\rho$ to an inertia group $I_{2}$ is described by two characters of level $2$ (as defined in \cite[Definition 3.3.2]{Cais2007}). We conclude from \cite[Definition 5.4.1]{Cais2007} that the optimal weight $k(\rho)$ is equal to $2$. This case is also known as the {\it{supersingular case}}, see \cite[Section 2]{SteinRibet1}. 

By Theorem \ref{Serreconjecture}, we now know that there is a normalized eigenform $f$ of level $11$ and weight $2$ such that $\overline{\rho}_{f,\lambda}$ (as given by Theorem \ref{RepresentationEigenform}) is isomorphic to our representation $\rho$. 
We are thus naturally led to consider the curves $X_{0}(11)$ and $X_{1}(11)$. The curve $X_{0}(11)$ is exactly our curve $A$ and by Lemma \ref{Isomorphism}, we find that the corresponding representation $\overline{\rho}_{A,2}$ coincides with $\overline{\rho}_{B,2}$. 
We thus see that our choice for $A$ is quite natural from the viewpoint of Serre's conjecture.

These ideas give a general starting point for finding normalized cusp forms that give reciprocity laws for number fields $L_{h}$ defined by a cubic polynomial $h(x)$ over $\mathbb{Q}$ with Galois group equal to $S_{3}$. 
We first 
determine the level $N(\rho)$ of the representation $\rho:G_{\mathbb{Q}}\rightarrow{\mathrm{GL}_{2}(\mathbb{F}_{2})}$ arising from the $2$-torsion of the elliptic curve $X$ given in Weierstrass form by
\begin{equation}
X:y^2=h(x).
\end{equation}
We then calculate the optimal weight $k(\rho)$ as in \cite[Section 5.4]{Cais2007} and \cite[Section 2]{SteinRibet1}. Using Theorem \ref{Serreconjecture}, we see that there exists a normalized eigenform $f$ of weight $k(\rho)$ and level $N(\rho)$ whose corresponding representation as in Theorem \ref{RepresentationEigenform} is equivalent to ours. Together with a version of quadratic reciprocity on the quadratic subfield (compare with the condition $p\equiv{1,3,4,5,9}\bmod{11}$ in Proposition \ref{ModularProp}), we then obtain a reciprocity law for primes $p$ to split in $L$ as in Proposition \ref{ModularProp}. 

\end{remark}

\begin{remark}\label{remprep}
{\it{(Interpretation in terms of representation theory)}}
The matrices appearing in Equation \ref{MatricesTrace} form two conjugacy classes. Using the relation $\mathrm{tr}(\rho(\mathrm{Frob}_{p}))\equiv{}0\bmod{2}$, we were able prove that the conjugacy class of a Frobenius automorphism equals one of these two classes. To be able to conclude that the conjugacy class of $\mathrm{Frob}_{p}$ is the conjugacy class of the identity matrix, we needed the extra condition afforded by the unique nontrivial quadratic character of $(\mathbb{Z}/11\mathbb{Z})^{*}$. 

We now interpret these relations in terms of representation theory. The group in question is the Galois group of $L/\mathbb{Q}$, which is isomorphic to $S_{3}$. This group has three irreducible representations, two of degree one and one of degree two. To determine the splitting behavior of a prime, it suffices to know the conjugacy class of a Frobenius automorphism above $p$. We do this using the concept of {\it{class functions}}. These are functions 
\begin{equation}
f:G\rightarrow{\mathbb{C}}
\end{equation}
such that $f(sts^{-1})=f(t)$ for every $s$ and $t$ in $G$, see \cite[Chapter 2]{Serre1977}. A basic result in representation theory then says that the characters of the irreducible representations form an orthonormal basis for the vector space of all class functions, see \cite[Page 19, Theorem 6]{Serre1977}. We can thus write every class function as a unique linear combination of these characters
\begin{equation}
f=\sum_{i}c_{i}\chi_{i}.
\end{equation}

In particular, we can write the indicator function of any conjugacy class $C_{i}\subset{G}$ as a linear combination of the characters of the irreducible representations.
Suppose now that we know $\chi_{i}(\mathrm{Frob}_{p})$ for every $i$. We can then evaluate the indicator functions $1_{C_{i}}$ of the conjugacy classes $C_{i}$ at $\mathrm{Frob}_{p}$. The resulting data then immediately tells us what conjugacy class 
$\mathrm{Frob}_{p}$ is in. 

There is a slight problem however in the connection between our representation $\rho:G\rightarrow{\mathrm{GL}_{2}(\mathbb{F}_{2})}$ and the theory mentioned above: this representation $\rho$ is in characteristic two. One can however embed $\mathrm{GL}_{2}(\mathbb{F}_{2})$ in $\mathrm{GL}_{2}(\mathbb{Z})$, which also gives an embedding into $\mathrm{GL}_{2}(\mathbb{C})$, see \cite[Chapter 9, Page 371]{ModularForms2016} for the embedding.

At any rate, we now clearly see why we needed the second condition on $p$ afforded by the quadratic character: it corresponds to the other nontrivial irreducible representation of $S_{3}$. Together with the irreducible representation of degree two, this completely determines the conjugacy class of the Frobenius automorphism. 
\end{remark}

\subsubsection*{Acknowledgements.} 
The authors would like to thank Associate Professor Burkard Polster (from {\it{"Mathologer"}}) for bringing these numbers under their attention through his video on YouTube on Tribonacci numbers. The authors would also like to thank Prof. Jaap Top for pointing out the connection with modular forms and the unique newform of weight $2$ and level $11$ and the referees for their comments and remarks, in particular for suggesting Lemmas \ref{VectorSpace} and \ref{Constants}. 
\bibliographystyle{alpha}
\bibliography{bibfiles}{}

\newcommand{\etalchar}[1]{$^{#1}$}
\begin{thebibliography}{ECdJ{\etalchar{+}}11}

\bibitem[Cai07]{Cais2007}
Bryden Cais.
\newblock Serre's conjectures.
\newblock {\em http://math.stanford.edu/~conrad/vigregroup/vigre05/Serre05},
  2007.

\bibitem[CBK{\etalchar{+}}46]{crelle1846journal}
A.L. Crelle, C.W. Borchardt, L.~Kronecker, L.~Fuchs, K.~Hensel, H.~Hasse, and
  F.~Schottky.
\newblock {\em Journal f{\"u}r die reine und angewandte Mathematik}.
\newblock Number v. 33-34. W. de Gruyter, 1846.

\bibitem[Dic99]{Dickson1}
Leonard~Eugene Dickson.
\newblock {\em History of the Theory of Numbers (3 Volumes)}.
\newblock American Mathematical Society, 1999.

\bibitem[DS16]{ModularForms2016}
Fred Diamond and Jerry Shurman.
\newblock {\em A First Course in Modular Forms (Graduate Texts in Mathematics,
  Vol. 228)}.
\newblock Springer, 2016.

\bibitem[ECdJ{\etalchar{+}}11]{CompGalois}
Bas Edixhoven, Jean-Marc Couveignes, Robin de~Jong, Franz Merkl, and Johan
  Bosman.
\newblock {\em Computational Aspects of Modular Forms and Galois
  Representations}.
\newblock Annals of Mathematics Studies. Princeton University Press, 2011.

\bibitem[KW09a]{Khare2009}
Chandrashekhar Khare and Jean-Pierre Wintenberger.
\newblock Serre's modularity conjecture ({I}).
\newblock {\em Inventiones mathematicae}, 178(3):485--504, jul 2009.

\bibitem[KW09b]{Khare20092}
Chandrashekhar Khare and Jean-Pierre Wintenberger.
\newblock Serre's modularity conjecture ({II}).
\newblock {\em Inventiones mathematicae}, 178(3):505--586, jul 2009.

\bibitem[RS18]{SteinRibet1}
Ken Ribet and William Stein.
\newblock Lectures on {S}erre's conjectures.
\newblock {\em https://wstein.org/papers/serre/ribet-stein.pdf}, 2018.

\bibitem[Ser56]{Serre1956}
Jean-Pierre Serre.
\newblock G{\'{e}}om{\'{e}}trie alg{\'{e}}brique et g{\'{e}}om{\'{e}}trie
  analytique.
\newblock {\em Annales de l'institut Fourier}, 6:1--42, 1956.

\bibitem[Ser77]{Serre1977}
Jean-Pierre Serre.
\newblock {\em Linear Representations of Finite Groups}.
\newblock Springer New York, 1977.

\bibitem[Ser87]{Ser1987}
Jean-Pierre Serre.
\newblock Sur les repr\'{e}sentations modulaires de degr\'{e} $2$ de
  $\mathrm{Gal}(\overline{\mathbf{{q}}}/\mathbf{Q})$.
\newblock {\em Duke Math. J. 54}, no. 1:179--230, 1987.

\bibitem[Shi66]{Shimura1966}
Goro Shimura.
\newblock A reciprocity law in non-solvable extensions.
\newblock {\em Journal f\"{u}r die reine und angewandte {M}athematik},
  221:209--220, 1966.

\bibitem[Shi71]{shimura1971}
Goro Shimura.
\newblock {\em Introduction to the Arithmetic Theory of Automorphic Functions}.
\newblock Kan{\^o} memorial lectures. Princeton University Press, 1971.

\bibitem[Sil09]{Silv1}
Joseph~H. Silverman.
\newblock {\em The Arithmetic of Elliptic Curves}.
\newblock Springer New York, 2009.

\bibitem[ST15]{Silverman2015}
Joseph~H. Silverman and John~T. Tate.
\newblock {\em Rational Points on Elliptic Curves}.
\newblock Springer International Publishing, 2015.

\bibitem[Ste02]{Ste1}
Peter Stevenhagen.
\newblock Class {F}ield {T}heory.
\newblock {\em http://websites.math.leidenuniv.nl/algebra/cft.pdf}, 2002.

\bibitem[Ste17]{Ste2}
Peter Stevenhagen.
\newblock Algebraic number theory.
\newblock {\em http://websites.math.leidenuniv.nl/algebra/ant.pdf}, 2017.

\bibitem[Wes18]{TWe1}
Tom Weston.
\newblock The modular curves ${X}_{0}(11)$ and ${X}_{1}(11)$.
\newblock {\em http://swc.math.arizona.edu/aws/2001/01Weston1.pdf}, 2018.

\end{thebibliography}

\end{document}